\documentclass[12pt]{amsart}
\usepackage{amsmath,amssymb,amsthm}
\usepackage{setspace}
\usepackage{moreverb}
\usepackage{color}
\usepackage{amscd}

\newcommand{\ignore}[1]{}

\newtheorem{thm}{Theorem}[section]
\newtheorem{lemma}[thm]{Lemma}
\newtheorem{prop}[thm]{Proposition}
\newtheorem{cor}[thm]{Corollary}

\theoremstyle{definition}
\newtheorem{nota}[thm]{Notation}
\newtheorem{defn}[thm]{Definition}
\newtheorem{prop-def}[thm]{Proposition-Definition}
\newtheorem{example}[thm]{Example}
\newtheorem{rmk}[thm]{Remark}

\numberwithin{equation}{section}

\input xy
\xyoption{all}

\newcommand{\QQ}{\mathbb Q}

\newcommand{\Z}{\mathbb{Z}}

\newcommand{\Oc}{\mathcal{O}}

\newcommand{\J}{\mathcal{J}}

\renewcommand{\geq}{\geqslant}
\renewcommand{\leq}{\leqslant}

\newcommand{\M}{{\mathfrak m}}
\usepackage{ mathrsfs }
\newcommand{\D}{\mathscr{D}}

\newcommand{\ak}{\mathfrak a}
\DeclareMathOperator{\Pic}{Pic}
\DeclareMathOperator{\Spec}{Spec}
\DeclareMathOperator{\lct}{lct}

\begin{document}
 
\title[]{Computing jumping numbers in higher dimensions}

\author[H. Baumers]{Hans Baumers}

\author[F. Dachs-Cadefau]{Ferran Dachs-Cadefau}

\address{KU Leuven\\ Department of Mathematics \\ Celestijnenlaan 200B box 2400\\
BE-3001 Leuven, Belgium}

\email{Hans.Baumers@wis.kuleuven.be, Ferran.DachsCadefau@wis.kuleuven.be}

\thanks{The first author is supported by a PhD fellowship of the Research Foundation - Flanders (FWO). The second author was partially supported by Generalitat de Catalunya SGR2014-634 project,
Spanish Ministerio de Econom\'ia y Competitividad
MTM2015-69135-P and by the KU Leuven grant OT/11/069.}

\begin{abstract}
The aim of this paper is to generalize the algorithm to compute jumping numbers on rational surfaces described in \cite{AAD14} to varieties of dimension at least 3. Therefore, we introduce the notion of $\pi$-antieffective divisors, generalizing antinef divisors. Using these divisors, we present a way to find a small subset of the `classical' candidate jumping numbers of an ideal, containing all the jumping numbers. Moreover, many of these numbers are automatically jumping numbers, and in many other cases, it can be easily checked. 
\end{abstract}

\maketitle

\section{Introduction}
To an ideal sheaf $\ak$ on a smooth algebraic variety $X$, one can associate its multiplier ideals $\mathcal J(X,\ak^c)$, where $c\in\QQ_{\geq0}$. They form a nested family of ideals in $\Oc_X$, which decreases when $c$ increases. The values of $c$ where the ideal changes are called the \emph{jumping numbers} of the pair $(X,\ak)$. They are very interesting geometric invariants, that were studied in \cite{ELSV04}, but also appeared earlier in \cite{Lib83}, \cite{LV90}, \cite{Vaq92} and \cite{Vaq94}. 
The jumping numbers determine in some sense how bad a singularity is. For example, if $\ak$ is the ideal corresponding to a smooth hypersurface, then the jumping numbers are just the positive integers. When the ideal represents a more singular variety - or when it takes more blow-ups to obtain a log resolution - the jumping numbers are in general smaller and more numerous.

The smallest jumping number is called the \emph{log canonical threshold}. Koll\'ar (see \cite{Kol97}) proved that, if $\ak$ is a principal ideal, it corresponds to the smallest root of the Bernstein-Sato polynomial. Ein, Lazarsfeld, Smith and Varolin (see \cite{ELSV04}) generalized this result for all jumping numbers in the interval (0,1]. 

We are interested in ways to compute jumping numbers. For monomial ideals, Howald (\cite{How01}) showed a combinatoric description of the multiplier ideal, which also allows to determine the jumping numbers.
In \cite{Tuc10}, Tucker presents an algorithm to compute jumping numbers on surfaces with rational singularities. Alberich-Carrami\~nana, \`Alvarez Montaner and the second author (\cite{AAD14}) introduce another algorithm in that setting. Shibuta (\cite{Shi11}) constructed an algorithm to compute multiplier ideals and jumping numbers in arbitrary dimensions using $\mathcal D$-modules, which was simplified by Berkesch and Leykin (\cite{BL10}).

In this paper, we present an algorithm that can be used for computing jumping numbers in arbitrary dimensions, based on the algorithm in \cite{AAD14}. The idea is to start with computing the so-called \emph{supercandidates}, and then checking whether they are jumping numbers. The supercandidates can be computed in arbitrary dimensions, as long as we have enough understanding of the Picard groups of the exceptional divisors in a chosen resolution of $\ak$. For checking that they are jumping numbers, we give some possible criteria. Although we do not present a technique that works in full generality, we are able to compute jumping numbers of ideals where previous algorithms seemed to be insufficient (or got stuck while computing them).

In Section \ref{sec:preliminaries} we introduce the basics on multiplier ideals and jumping numbers, together with some elementary results that we need. We also recall some of the notions introduced by Tucker in \cite{Tuc10}. 

In Section \ref{sec:unloading}, we define the $\pi$-antieffective divisors, which are a generalization of the notion of antinef divisors. We also construct a method to compute the $\pi$-antieffective closure of a divisor, generalizing the unloading procedure presented in \cite{Lau72} and \cite{Enr15}. Lipman's correspondence between antinef divisors and integrally closed ideals (see \cite{Lip69}) does not hold anymore in higher dimensions, but we present a weaker alternative in Section \ref{ssec:correspondence}.

Section \ref{sec:algorithm} contains the core of the paper. Here we present our algorithm to compute the \emph{supercandidates}, and ways to check whether they are jumping numbers.

Finally, in Section \ref{sec:examples} we present some illustrative examples.

\section{Preliminaries}\label{sec:preliminaries}
Through this section, let $X$ be a smooth algebraic variety with $\dim X=n\geqslant 2$ over an algebraically closed field $k$ of characteristic zero. 
Let $\ak$ be a sheaf of ideals on $X$. We define a \emph{log resolution} of the pair $(X,\ak)$ as a birational morphism $\pi:Y\to X$, such that

\begin{itemize}
 \item $Y$ is smooth,
 \item the pre-image of $\ak$ is locally principal, i.e., $\ak \cdot \Oc_{Y}=\Oc_{Y}(-F)$ for some effective Cartier divisor $F$, and 
 \item $F+\text{Exc}(\pi)$ is a simple normal crossing divisor.
\end{itemize}
By Hironaka's resolution of singularities (see \cite{Hir64}), log resolutions always exist.
Note that if $\ak=\Oc_X(-D)$ for an effective divisor $D$, then $F=\pi^*D$, and $\pi$ will be called a log resolution of $(X,D)$ instead of $(X,\Oc_X(-D))$.

Given a birational morphism $\pi: Y\to X$, the relative canonical divisor measures in some way the difference between $X$ and $Y$.
\begin{defn}
Let $\pi:Y\to X$ be a birational morphism of smooth varieties, the \emph{relative canonical divisor} of $\pi$ is the divisor class
\[K_\pi := K_{Y} - \pi^*K_X\,.\]
\end{defn}
It is important to notice that, even though $K_X$ and $K_Y$ are only defined as divisor classes, the relative canonical divisor can be chosen to be an effective divisor, supported on the exceptional locus of $\pi$. Indeed, if $X$ and $Y$ are smooth, there is a unique way to write \[K_\pi = \sum k_i E_i\,,\] with $k_i\in \Z_{\geq 0}$, where the $E_i$ are the irreducible components of $\text{Exc}(\pi)$.

%

Another notion that we need to introduce is $\QQ$-divisors.

\begin{defn}
A \emph{$\QQ$-divisor} $D$ on an algebraic variety $Y$ is a formal finite sum $D=\sum a_i D_i$, where the $D_i$ are irreducible codimension one subvarieties of $Y$, and $a_i\in \QQ$.
\end{defn}

For any $\QQ$-divisor $D=\sum a_iD_i$, one denotes its \emph{round-down} and \emph{round-up} as
\[\lfloor D\rfloor = \sum \lfloor a_i \rfloor D_i\,\text{\hspace{1cm} and \hspace{1cm}}\lceil D \rceil = \sum \lceil a_i \rceil D_i\,,\] respectively.

\subsection{Multiplier Ideals}
Having introduced these basic notions, we define multiplier ideals.


\begin{defn}
Let $\ak$ be a sheaf of ideals on $X$, $\pi:Y\to X$ a log resolution of $\ak$, and $F$ the divisor satisfying $\ak\cdot\Oc_{Y}= \Oc_{Y}(-F)$.
For $c \in \QQ_{\geq0}$, we define the \emph{multiplier ideal associated to $\ak$ with coefficient $c$} as
\[\mathcal J(X,\ak^c) := \pi_*\Oc_{Y}\left(K_\pi - \lfloor c F\rfloor\right),\]
where $K_\pi$ is the relative canonical divisor.
\end{defn}

If $\ak=\Oc_X(-D)$ for an effective divisor $D$ on $X$, we will denote the multiplier ideals by $\mathcal J(X,cD)$.
For simplicity, if no confusion can arise, we will simply write $\mathcal J(\ak^c)$ or $\mathcal J(c D)$.

\begin{rmk}
 It is clear from the definition of $K_\pi$ that $\pi_*\Oc_Y(K_\pi)=\Oc_X$, and therefore for any effective divisor $N$, we have \[\pi_*\Oc_Y(K_\pi-N)\subseteq \Oc_X\,.\] It is due to this fact that $\J(\ak^c)\subseteq \Oc_X$ are subsheaves of $\Oc_X$, which justifies the name \emph{multiplier ideal}.
\end{rmk}

It is easy to see that $\J(\ak^c)\subseteq \J(\ak^{c'})$ if $c\geq c'$, and that $\J(\ak^{(c +\varepsilon)})=\J(\ak^c)$ if $0\leq\varepsilon\ll 1$. This yields the following result.
\begin{prop-def}
Let $\ak$ be an ideal sheaf on $X$. There exists an increasing sequence of rational numbers
\[0=\lambda_0<\lambda_1<\lambda_2<\lambda_3<\dots\]
satisfying
\begin{itemize}
\setlength\itemsep{0.1em}
\item $\J(\ak^{\lambda_i})\varsupsetneq \J(\ak^{\lambda_{i+1}})$ for $i \in \mathbb N$,
\item $\J(\ak^c) = \J(\ak^{\lambda_i})$ for $c\in [\lambda_i,\lambda_{i+1})$.
\end{itemize}
The numbers $\lambda_i$, $i>0$,  are called the \emph{jumping numbers} of $\ak$.
\end{prop-def}
The jumping numbers of a divisor $D$ are defined analogously. The smallest jumping number is called the \emph{log canonical threshold} of $\ak$ or $D$, and is denoted by $\lct(X,\ak)$ or $\lct(X,D)$, respectively. This is a very important invariant of the pair $(X,\ak)$ or $(X,D)$, that appears in different branches of algebraic geometry. For a nice overview, we refer to \cite{Kol97}.

\vspace{3mm}


Now we fix some notations. For a pair $(X,\ak)$, we fix a log resolution $\pi:Y\to X$. We denote by $F$ the divisor satisfying $\ak\cdot\Oc_Y=\Oc_Y(-F)$. The irreducible components of $F$ are denoted $E_i$, $i\in I$, and we write \[F=\sum_{i\in I} e_iE_i\,,\] where $e_i\in\Z_{>0}$. The divisor $F$ has an exceptional part and a non-exceptional part, also called the \emph{affine} part. The affine part is sometimes denoted $F_{aff}$, and the exceptional components are denoted $E_1,\dots,E_r$. So we also have \[F=F_{aff} + \sum_{i=1}^r e_i E_i\,.\]
Note that $F_{aff} = 0$ whenever the support of $\ak$ has codimension at least 2.

\vspace{3mm}

Let $\ak\subseteq \Oc_X$ be an ideal on $X$, $D$ be an effective divisor on $X$, $F$ the divisor on $Y$ defined as before and $c$ a positive rational number. The multiplier ideals associated to $\ak$ or $D$ and $c$ satisfy the following properties (see \cite{LazII04}).

\begin{itemize}
 \item The definition of multiplier ideal does not depend on the resolution (Esnault-Viehweg in \cite{EV92}).
 \item (Local Vanishing) $R^i\pi_*\Oc_{Y}(K_{\pi}-\lfloor cF\rfloor)=0$ for all $i>0$ and $c>0$.
 \item We have that $\J(\ak^c)$ is integrally closed for all $c>0$.
 \item The integers are jumping numbers for the pair $(X,D)$.
 \item For $c>0$, we have that
 \begin{align*}
  \J((c+1) D) 	&= \pi_*\Oc_{Y}(K_\pi-\lfloor c \pi^* D\rfloor - \pi^* D)\\
			&= \J(c D)\otimes_{\Oc_X} \Oc_X(-D).
 \end{align*}
It follows that $\lambda>0$ is a jumping number if and only if $\lambda +1$ is a jumping number.
 \item (Skoda's theorem) If $m\in\mathbb N$ with $m\geqslant n$, then
\[\J(\ak^m) = \ak\J(\ak^{m-1})\,.\]
Therefore, for any $\lambda \geq n$, one has that $\lambda$ is a jumping number if and only if $\lambda-1$ is a jumping number.
 \item From the proof of Skoda's theorem, one can actually deduce a stronger result. If $\ak$ is an ideal generated by $\ell$ elements and $m\geqslant \ell$, then
\[\J(\ak^m) = \ak\J(\ak^{m-1})\,.\]
Therefore, for any $\lambda \geq \ell$, one has that $\lambda$ is a jumping number if and only if $\lambda-1$ is a jumping number.
\end{itemize}

\begin{rmk}
 Lipman and Watanabe (see \cite{LW03}), and independently Favre and Jonsson (see \cite{FJ05}), proved that every integrally closed ideal in a two-dimensional regular local ring is a multiplier ideal. However, this is no longer true in higher dimensions. Lazarsfeld and Lee showed in \cite{LL07} that if $\dim X\geqslant 3$, integrally closed ideals need to satisfy certain conditions in order to be realized as multiplier ideals. The conditions allow them to give examples of integrally closed ideals that cannot be realized as multiplier ideals.
\end{rmk}

\subsection{Contributing divisors}

One can easily see from the definition of multiplier ideals that, with the notations above, the jumping numbers are contained in the set
\[\left\{\left.\frac{k_i+n}{e_i}\right| i\in I, n\in \Z_{>0}\right\}.\]
These numbers are the \emph{candidate jumping numbers}. It is important to notice that the smallest candidate is always a jumping number, and hence it equals the log canonical threshold. So we have \[\lct(X,\ak)=\min\left\{\left.\frac{k_i+1}{e_i}\right|i\in I\right\},\] and similar for a divisor $D$.

\vskip 3mm

Furthermore, if $E_i$ is not exceptional, then $k_i=0$ and the candidates \[\left\{\left.\frac{m}{e_i} \hskip 2mm \right| \hskip 2mm m\in \Z_{>0}  \right\}\] are always jumping numbers. For the exceptional ones, in general many candidate jumping numbers are not a jumping number. 

\begin{defn}
Let $G$ be a reduced divisor supported on the exceptional part of $\pi$ and $\lambda$ a positive rational number. We will say that $\lambda$ is a \emph{candidate jumping number for $G$} if and only if $\lambda$ can be expressed as $\frac{k_i+n_i}{e_i}$ for each $E_i\leq G$ with $n_i\in \Z_{>0}$.
\end{defn}

A notion that is stronger and more interesting than being a candidate for a divisor, is being contributed by a divisor. This notion was introduced by Smith and Thompson in \cite{ST07}, and developed further by Tucker in \cite{Tuc10}.

\begin{defn}\cite[Definition 3.1]{Tuc10}
Let $G$ be a reduced divisor supported on the exceptional part of $\pi$, and $\lambda$ a candidate jumping number for $G$. We say that $G$ \emph{contributes} $\lambda$ as a jumping number if
\[ \pi_*\Oc_Y( K_\pi -\lfloor\lambda F\rfloor+G)\varsupsetneq \mathcal J(X,\ak^\lambda)\,.\]
We will say that this contribution is \emph{critical} if moreover for any non-zero divisor $G'< G $ one has
\[ \pi_*\Oc_Y( K_\pi -\lfloor\lambda F\rfloor+G')= \mathcal J(X,\ak^\lambda)\,.\]
\end{defn}

As an illustration of these concepts, we consider some examples in dimension two.

\begin{example}
Let $X$ be the affine plane and $D=\{y^2=x^3\}$. Let $\pi:Y\to X$ be the minimal log resolution of $D$, and $E_1$, $E_2$ and $E_3$ the exceptional divisors. Then we have $\pi^*D = D_{aff} + 2E_1+3E_2+6E_3$, where $D_{aff}$ is the strict transform of $D$. Moreover, we have that $K_\pi = E_1+2E_2+4E_3$, so the candidate jumping numbers are
\[\left\{\left.\frac{0+n}{1},\frac{1+n}{2}, \frac{2+n}{3}, \frac{4+n}{6} \right| n\in \Z_{>0}\right\}\,.\]
Hence \[\lct(X,D)=\frac{5}{6}\] is the smallest jumping number. Moreover, since we are in the case of a divisor, the integers are always jumping numbers, and the jumping numbers are periodic. Then one concludes that \[\left\{\frac{5}{6},1,\frac{11}{6},2,\frac{17}{6},3,\dots\right\}\] is the set of jumping numbers. Clearly, all jumping numbers of the form $\frac 56+m$ for $m\in\mathbb N$ are contributed by $E_3$, and all integers are contributed by $D_{aff}$.
\end{example}


\begin{example}\cite[Example 3.9]{AAD14}
Let $X$ be the affine plane again and consider the ideal $\ak = (x^2y^2,x^5,y^5,xy^4,x^4y)\subseteq\Oc_X$. Let $\pi:Y\to X$ be its minimal log resolution. Then $K_\pi = E_1+2E_2+4E_3+2E_4+4E_5$, and $\ak\cdot \Oc_{Y}=\Oc_{Y}(-F)$, where $F=4E_1+5E_2+10E_3+5E_4+10E_5$. 
Using the fact that the log canonical threshold is the minimal candidate jumping number, one can see that $\lct(X,\ak)=\frac{1}{2}$. In order to compute all the jumping numbers in this case, we use the algorithm presented in \cite{AAD14}. This yields that the set of jumping numbers associated to $\ak$ is \[\left\{\left.\frac{1}{2}, \frac{7}{10}, \frac{n}{10}\right| n\geq 9\right\}\,.\] 
We can see that $\frac{7}{10}$ is contributed by $E_3+E_5$. However, since it is (critically) contributed by $E_3$ and also by $E_5$, $\frac{7}{10}$ is not critically contributed by $E_3+E_5$.
\end{example}

The following result is a nice characterization of contribution which will be used in the following sections. It appears in \cite{ST07} and \cite[Proposition 4.1]{Tuc10} for surfaces, but it holds in a more general setting. We repeat it here for completeness. First, we introduce a notation.

\begin{nota}
If $E$ is a subscheme of a scheme $Y$, $\iota:E\to Y$ is the embedding, and $\mathcal F$ is a sheaf of $\Oc_Y$-modules, then we denote by $\mathcal F|_E$ the $\Oc_E$-module $\iota^*\mathcal F$. Sometimes, if $\mathcal G$ is a sheaf of $\Oc_E$-modules, we will consider it as a sheaf on $Y$ by simply writing $\mathcal G$ instead of $\iota_*\mathcal G$.
\end{nota}

\begin{prop}\label{prop:contribution_globalsections}
Suppose $\lambda$ is a candidate jumping number for the reduced divisor $G$. Suppose that $G$ is mapped onto an affine subscheme of $X$.  Then $\lambda$ is realized as a jumping number for $(X,D)$ or $(X,\ak)$ contributed by $G$ if and only if
\[H^0(G,\Oc_Y(K_\pi-\lfloor \lambda F\rfloor+G)|_G)\neq 0\,.\]
Furthermore, this contribution is critical if and only if we have
\[H^0(G',\Oc_Y(K_\pi-\lfloor \lambda F\rfloor + G')|_{G'})=0\] for all divisors $G'$ on $Y$ with $0\leq G' < G$.
\end{prop}
\begin{proof}
Consider the exact sequence
\[0\to \Oc_Y(K_\pi-\lfloor \lambda F\rfloor) \to \Oc_Y(K_\pi-\lfloor\lambda F\rfloor +G) \to \Oc_Y(K_\pi-\lfloor \lambda F\rfloor+G)|_G\to 0\,.\]
After pushing forward through $\pi$, we get
\[0\to\mathcal J(X,\ak^\lambda)\to \pi_*\Oc_Y(K_\pi-\lfloor \lambda F\rfloor+G)\to \pi_*(\Oc_Y(K_\pi-\lfloor \lambda F\rfloor +G)|_G)\to 0\,,\]
since by local vanishing we have $R^1\pi_*\Oc_Y(K_\pi-\lfloor \lambda F\rfloor)=0$. So we see that $\lambda$ is a jumping number contributed by $G$ if and only if $\pi_*(\Oc_Y(K_\pi-\lfloor \lambda F\rfloor +G)|_G)\neq0$. Since $G$ is mapped onto an affine scheme, this is equivalent to $H^0(G,\Oc_Y(K_\pi-\lfloor\lambda F\rfloor+G)|_G)\neq 0$. The second statement follows immediately from the definition of critical contribution.
\end{proof}
\begin{rmk}
The condition that $G$ is mapped onto an affine subscheme of $X$ is a generalization of the two-dimensional case, where all the exceptional divisors are contracted to a point, and is also sufficient for our purposes, since we will only consider affine varieties $X$.
\end{rmk}

\begin{cor}[{\cite[Corollary 4.2]{Tuc10}}]\label{cor:connection_critical}
If $G$ critically contributes a jumping number $\lambda$, then $G$ is connected.
\end{cor}
\begin{proof}
Suppose that $G$ is disconnected, so $G=G'+G''$ with $0<G',G''<G$ and $G'$ and $G''$ disjoint. Then
\begin{multline*}
 H^0(G,\Oc_Y(K_\pi-\lfloor \lambda F\rfloor+G)|_G) \\= H^0(G',\Oc_Y(K_\pi-\lfloor \lambda F\rfloor+G')|_{G'}) \oplus H^0(G'',\Oc_Y(K_\pi-\lfloor \lambda F\rfloor+G'')|_{G''})\,.
\end{multline*}
So if $\lambda$ is contributed by $G$, it is also contributed by $G'$ or $G''$, contradicting critical contribution.
\end{proof}

\section{$\pi$-antieffective divisors and integrally closed ideals}\label{sec:unloading}
From now on, we consider a regular local ring $R$ over $k$ of dimension at least $2$, such that $X=\Spec R$ is the germ of a smooth algebraic variety over $k$. Let $\ak$ be an ideal sheaf on $X$, fix a log resolution $\pi:Y\to X$, and let $F$ be the divisor satisfying $\ak\cdot \Oc_Y = \Oc_Y(-F)$. We denote the relative canonical divisor by $K_\pi$, this divisor is supported over the exceptional divisors $E_i$ for $i=1,...,r$.  

If $\dim X=2$, recall the following definition.
\begin{defn}
If $Y$ is a surface, then a divisor $D$ on $Y$ is called \textit{antinef} (or \textit{$\pi$-antinef}) if $-D\cdot E_i\geq 0$ for all $i\in \{1,\dots,r\}$.
\end{defn}
This notion is introduced in \cite{Lip69}, and is also explained in \cite{Tuc09}. A generalization of this concept to higher dimensions is given in \cite{CGL96}. 

In section \ref{ssec:antinef_closure}, we define $\pi$-antieffective divisors, which is a generalization to higher dimensions of antinef divisors. It is different from the one in \cite{CGL96}, but more useful for our purposes. We prove the existence of the $\pi$-antieffective closure and present a way to compute it.

Lipman proved that in the two-dimensional case there is a one-to-one correspondence between antinef divisors on $Y$ and integrally closed ideals in $R$ defining invertible sheaves on $Y$. In higher dimensions, this correspondence does not hold anymore. We will prove a weaker version in Section \ref{ssec:correspondence}.

Before introducing $\pi$-antieffective divisors, we give a definition.

\begin{defn}
 We say that two divisors $D_1$ and $D_2$ on $Y$ are \textit{equivalent} if and only if they define the same ideal in $R$, i.e., if and only if $\pi_*\Oc_Y(-D_1)=\pi_*\Oc_Y(-{D_2})$.
\end{defn}

From now on, if we want to refer to linear or numerical equivalence, we will state it clearly, so no confusion will arise.

\subsection{Unloading}\label{ssec:antinef_closure}
The following definition is the generalization we want of the notion of antinef divisor.

\begin{defn}
 Let $D$ be a divisor on $Y$ with integral coefficients. We say that $D$ is \emph{$\pi$-antieffective} if and only if $H^0(E,\Oc_{Y}(-D)|_E)\neq 0$ for every $\pi$-exceptional prime divisor $E$, i.e., if and only if $-D|_{E}$ is a divisor class on $E$ containing an effective divisor.
\end{defn}

In general, for any divisor, we can find a $\pi$-antieffective divisor equivalent to the given divisor.

\begin{thm}\label{Thm:unloading}
  Let $D$ be a divisor on $Y$, then there exists a unique integral effective $\pi$-antieffective divisor $\tilde{D}$ satisfying
 \begin{itemize}
  \item $\tilde{D}\geqslant D$, and
  \item for any $\pi$-antieffective divisor $D'$ such that $D'\geqslant D$, we have $D'\geqslant \tilde{D}$.
 \end{itemize}
 Moreover, this divisor is equivalent to $D$. 
\end{thm}

This leads us to the following definition.


\begin{defn}
The $\pi$-antieffective divisor $\tilde D$ satisfying
 \begin{itemize}
  \item $\tilde{D}\geqslant D$, and
  \item for any $\pi$-antieffective divisor $D'$ such that $D'\geqslant D$, we have $D'\geqslant \tilde{D}$,
 \end{itemize}
is called the \textit{$\pi$-antieffective closure} of $D$.
\end{defn}

The proof of the theorem will be divided in several results. In the forthcoming lemma, we prove that such a minimal divisor exists. Later on (see Propositions \ref{prop:eq_divisors} and \ref{prop:fin_num_steps}), we prove that a divisor and its $\pi$-antieffective closure are equivalent.
\newpage
\begin{lemma}\label{existence_antieffective_closure}
 Let $D$ be a divisor, then there exists a unique integral $\pi$-antieffective divisor $\tilde{D}$ satisfying
 \begin{itemize}
  \item $\tilde{D}\geqslant D$, and
  \item for any $\pi$-antieffective divisor $D'$ such that $D'\geqslant D$, we have $D'\geqslant \tilde{D}$.
 \end{itemize}
Moreover, $\tilde D-D$ is supported on the exceptional locus of $\pi$.
\end{lemma}

\begin{proof}
This proof is based partially on Paragraph 2.2 in \cite{Tuc09}.

Let $\D$ be the set of all $\pi$-antieffective divisors $D'$ such that $D'\geqslant D$. In the first part of the proof, we show that $\D$ is non-empty, while the second part is devoted to prove that $\D$ has a unique minimal element.
 
 

We start with showing that $\D$ is non-empty. Take $g\in R$ such that $\nu_i(g)>0$ for all divisorial valuations $\nu_i$ associated to one of the exceptional divisors $E_i$. Let $G$ be the exceptional part of $\textrm{div}(g)$, and take $D_0:=\pi^*\pi_*D+mG$ for sufficiently large $m\in\mathbb N$. We claim that $D_0\in\D$. Clearly $D_0\geq D$. Moreover, $D_0$ is $\pi$-antieffective. Indeed, it equals $\pi^*(\pi_*D+m\text{div}(g)) - m G_{aff}$, where $G_{aff}$ is the affine part of $\pi^*\text{div}(g)$, and hence, since $\text{Pic }X = 0$, we have that $-D_0|_E$ is linearly equivalent to $mG_{aff}|_E$, which is clearly effective.

In order to check the unicity of a minimal element in $\D$, assume that there exist two different minimal divisors $D_1=\sum_i d_i^1 E_i$ and $D_2=\sum_i d_i^2 E_i$ in $\D$. Now, define $D'=\sum_i d_iE_i$ with $d_i=\min\{d_i^1,d_i^2\}$. Take an exceptional divisor $E$, and suppose that $d_E^1 \leq d_E^2$. Then
\[-D'|_E = -d_E^1 E|_E - \sum_{E_i\neq E} d_i E_i|_E\,.\]
If $E_i\neq E$, then $E_i|_E$ is an effective divisor on $E$, and therefore $\sum_{E_i\neq E} (d_i^1-d_i) E_i|_E = D_1|_E-D'|_E$ is an effective divisor on $E$. Since $D_1$ is $\pi$-antieffective, $-D_1|_E$ defines the class of an effective divisor, so also $-D'|_E$ defines an effective divisor class. By repeating this argument for any $E$, we conclude that $D'$ is also $\pi$-antieffective and satisfies $D'\leqslant D_1$ and $D'\leqslant D_2$, contradicting the minimality in $\D$ of $D_1$ and $D_2$.

Finally, since $D_0-D$ is supported on the exceptional locus of $\pi$, where $D_0$ is the divisor constructed above, the same holds for $\tilde D-D$, where $\tilde D$ is the unique minimal element of $\D$.
\end{proof}

The following result tells us how to find equivalent divisors.

\begin{prop}\label{prop:eq_divisors}
Let $D$ be a divisor on $Y$ and $E$ an exceptional divisor of $\pi$. 
If $H^0(E,\Oc_{Y}(-D)|_E)=0$, then \[\pi_*\Oc_{Y}(-D-E)=\pi_*\Oc_{Y}(-D)\,.\]
\end{prop}

\begin{proof}
Denote by $\iota:E\to Y$ the embedding of $E$ in $Y$.
Consider the short exact sequence
\[0\to \Oc_{Y}(-E)\to \Oc_{Y} \to \iota_{*}\Oc_E\to 0\,.\]
After tensoring with $\Oc_{Y}(-D)$ and pushing forward through $\pi$, we get the exact sequence
\begin{align}
0\to \pi_*\Oc_{Y}(-D-E)\to \pi_*\Oc_{Y}(-D) \to \pi_*(\iota_{*}\Oc_E\otimes_{\Oc_{Y}}\Oc_{Y}(-D))\,.\label{ses}
\end{align}
By the projection formula, we know that
\begin{align*}
\iota_{*}\Oc_E\otimes_{\Oc_{Y}}\Oc_{Y}(-D) &= \iota_{*}(\Oc_E\otimes_{\Oc_E} \iota^*\Oc_{Y}(-D))\\
&= \iota_{*}\iota^*\Oc_{Y}(-D)\,.
\end{align*}
Denote the image of $E$ through $\pi$ by $C$, and name the morphisms as follows.
\[\begin{CD}
E @>\iota>> Y\\
@V\pi'=\pi|_{E}VV @V\pi VV\\
C @>\iota_C>> X
\end{CD}\]
Then we obtain that the last sheaf in (\ref{ses}) equals $\iota_{C*}\pi'_*\iota^*\Oc_{Y}(-D)$. Since $\Oc_{Y}(-D)$ is invertible, we know that $\pi'_*\iota^*\Oc_{Y}(-D)$ is a quasi-coherent $\Oc_C$-module (see \cite[Proposition II.5.8]{Har77}). Hence, since $C$ is affine, this $\Oc_C$-module is determined by its global sections. Since by assumption $\iota^*\Oc_Y(-D)$ has no global sections, also $\iota_{C*}\pi'_*\iota^*\Oc_{Y}(-D)$ has no global sections so we conclude that the sheaf equals to the zero sheaf. Hence by (\ref{ses}), \[\pi_*\Oc_{Y}(-D-E)=\pi_*\Oc_{Y}(-D)\,.\]
\end{proof}


This result gives a constructive way to find the $\pi$-antieffective closure of a divisor $D$, called the {\bf unloading procedure.} It goes as follows. Let $D$ be a non-$\pi$-antieffective divisor. Then there exists at least one $\pi$-exceptional divisor $E$ such that $H^0(E,\Oc_Y(-D)|_E)=0$. We replace $D$ by $D':=D+E$ and repeat until the obtained divisor is $\pi$-antieffective. 

Note that we cannot accidentally `miss' the $\pi$-antieffective closure by chosing a specific order, since the order of adding $E_i$'s does not matter. Indeed, if $-D|_{E_1}$ and $-D|_{E_2}$ are both not linearly equivalent to an effective divisor, then $-(D+E_1)|_{E_2} = -D|_{E_2} -E_1|_{E_2}$ is not an effective divisor class either, so $E_2$ still has to be added somewhere in the process. Therefore, we obtained the following proposition.

\begin{prop}\label{prop:fin_num_steps}
 Let $D$ be a divisor, then after finitely many steps of unloading we reach the $\pi$-antieffective closure of $D$.
\end{prop}


Together with the fact that we only encounter equivalent divisors during the unloading procedure, this finishes the proof of Theorem \ref{Thm:unloading}. 

\vspace{3mm}

Our unloading procedure is based on work of Enriques in \cite{Enr15}. It is also Laufer's algorithm to compute the fundamental cycle \cite{Lau72}. It has also been described by Casas-Alvero \cite{Cas00} and Reguera \cite{Reg97}. An improved version of it is used in \cite{AAD14} to compute jumping numbers on surfaces with rational singularities. Here, we generalized the algorithm in \cite{Lau72} to higher dimensions. The main difference is that checking positivity of an intersection number is replaced by checking effectivity of a divisor class. In Section \ref{sec:algorithm}, the unloading procedure will be used in an algorithm to compute jumping numbers. 


\subsection{A correspondence between globally generated invertible sheaves and integrally closed ideals}\label{ssec:correspondence} 
The main goal of this section is to generalize the results of Lipman about the correspondence between integrally closed ideals and antinef divisors (see \cite[\textsection 18]{Lip69}). Lipman proves that in the two-dimensional case there is a one-to-one correspondence between antinef divisors and $\M$-primary integrally closed ideals that determine invertible sheaves on $Y$.

\vspace{2mm}

A first generalization of this result to higher dimensions is \cite[Proposition 1.20]{CGL96}. In this paper, the authors prove a similar relation between finitely supported integrally closed ideals and globally generated divisors on varieties obtained by finitely many point blow-ups. %
They also prove that if $-D$ is globally generated, then $-D\cdot C\geqslant 0$ for any exceptional curve $C$, as well as a counterexample for the reverse implication in dimensions higher than 2. We will prove a similar relation, which works for ideals that are not necessarily finitely supported, but which determine an invertible sheaf in a fixed birational morphism. The proof is essentially the same as the corresponding part of Lipman's proof in the two-dimensional setting. We repeat it here for completeness.

\begin{thm}
The mapping $D\mapsto I_D=\Gamma(Y,\Oc_{Y}(-D))$ is a one-to-one correspondence between the set of effective divisors $D$ on $Y$ such that $\Oc_Y(-D)$ is generated by its global sections, and integrally closed ideals $I$ of $R$ such that $I\cdot \Oc_Y$ is an invertible sheaf. The inverse is given by $I\mapsto D_I$, where $D_I$ is the divisor satisfying $I\cdot \Oc_Y = \Oc_Y(-D_I)$.
\end{thm}
\begin{proof}
If $D$ is an effective divisor, then $D=\sum_{i} d_i E_i$, where $E_i$ runs over all prime divisors on $Y$. Moreover, $d_i\geq0$ for all $i$, and $d_i=0$ for all but finitely many $i$. We have $\Oc_Y(-D)\subseteq \Oc_Y$, and hence
\[I_D=\Gamma(Y,\Oc_Y(-D))\subseteq \Gamma(Y,\Oc_Y)=R\,.\]
So $I_D$ is an ideal of $R$. Moreover, since $\Oc_Y(-D)$ is generated by its global sections, $I_D\cdot \Oc_Y = \Oc_Y(-D)$ is invertible.

Now we prove that $I_D$ is integrally closed. It is clear that
\[I_D=\{f\in R \mid \forall i: \nu_i(f)\geq d_i\}\,,\]
where $\nu_i$ is the divisorial valuation corresponding to $E_i$. Take $f\in R$ and suppose $f$ satisfies an equation $$f^n+a_1f^{n-1}+\dots+a_{n-1}f+a_n=0\,,$$ where $a_j\in (I_D)^j$ for $j=1,\dots,n$. Then for any $i$ the properties of valuations yield
\[n\nu_i(f)\geq \min_j\{\nu_i(a_j)+(n-j)\nu_i(f)\}\,.\]
So there exists $j_0\in\{1,\dots,n\}$ such that
\[n\nu_i(f)\geq\nu_i(a_{j_0})+(n-j_0)\nu_i(f)\,,\]
and hence $j_0\nu_i(f)\geq \nu_i(a_{j_0})\geq j_0d_i$, meaning that $f\in I_D$. This implies that $I_D$ is integrally closed.

Conversely, take an integrally closed ideal $I$ such that $I\cdot \Oc_Y$ is invertible. Then $D_I$ is such that $I\cdot\Oc_Y = \Oc_Y(-D_I)$, so $\Oc_Y(-D_I)$ is clearly generated by its global sections. Indeed, a set of generators of $I$ determines global sections of $I\cdot \Oc_Y$, and their restrictions generate the stalks.


By \cite[Proposition 6.2]{Lip69}, $I_{D_I} = \Gamma(Y,\Oc_Y(-D_I)) = I$, since $I$ is integrally closed. Conversely, if $\Oc_Y(-D)$ is generated by its global sections, then $\Oc_Y(-D) = I_D\cdot \Oc_Y$, and also $I_D\cdot \Oc_Y = \Oc_Y(-D_{I_D})$, which implies $D_{I_D} = D$.
\end{proof}

In our setting we are interested in the relation between $\pi$-antieffective divisors and integrally closed ideals. However, no one-to-one correspondence between them is known. We do know that the set of $\pi$-antieffective divisors contains the set of divisors associated to an integrally closed ideal.

\begin{prop}\label{prop:globallygenerated_implies_antieffective}
If $\Oc_Y(-D)$ is generated by global sections, then $D$ is $\pi$-antieffective.
\end{prop}
\begin{proof}
Let $E\subset Y$ be an exceptional divisor, and denote $\iota:E\to Y$ and $\mathcal L=\Oc_Y(-D)$. We have to show that $H^0(E,\iota^*\mathcal L)\neq 0$. Since $\mathcal L$ is globally generated, there exists an exact sequence \[0\to \mathcal K \to \Oc_Y^m \to \mathcal L \to 0\] of sheaves on $Y$. Pulling back by $\iota$, we get an exact sequence \[\iota^*\mathcal K \to \Oc_E^m \to \iota^* \mathcal L \to 0\] on $E$, so $\iota^*\mathcal L$ is also generated by its global sections. In particular, $H^0(E,\iota^*\mathcal L)\neq0$.
\end{proof}




The converse of the previous proposition is true in dimension $2$ (see \cite[\textsection 18]{Lip69}), but in higher dimensions it does not hold anymore, as is clear from the following example, which is inspired heavily on Remark 1.24 in \cite{CGL96}. 

\begin{example}
Let $X=\text{Spec}\,R$ be a smooth affine three-dimensional variety.
Consider the blowing-up at a point $0$ on $X$, followed by blowing up at nine points on the exceptional divisor $E_0$ in general position. This means that they lie on a non-singular cubic curve $C_0$, and that $C_0$ is the only cubic curve on $E_0$ passing through these nine points. Denote the new exceptional divisors by $E_1, \dots E_9$, and the morphism by $\pi:Y\to X$. We will show that $D:=3E_0 + 4\sum_{i=1}^9 E_i$ is $\pi$-antieffective, but that $\Oc_Y(-D)$ is not generated by its global sections.

Since $-D|_{E_0}$ is linearly equivalent to the strict transform of the curve $C_0$, it is effective in $\Pic E_0$.
Furthermore, $-D|_{E_i}$ for $i\in\{1,\dots,9\}$ is the class of a line on $E_i$, hence is also effective in $\Pic E_i$. So we see that $D$ is $\pi$-antieffective.


If $\Oc_Y(-D)$ would be globally generated, then also its restriction to $E_0$ should be globally generated, as is clear from the proof of Proposition \ref{prop:globallygenerated_implies_antieffective}. But this is the sheaf defined by the strict transform of the curve $C_0$. Since this is the only cubic curve on $\mathbb P^2$ through the nine points, this divisor cannot be moved, and hence the sheaf is not globally generated.
\end{example}

\section{An algorithm to compute jumping numbers}\label{sec:algorithm}
In this section, we discuss a technique that can be used to compute jumping numbers, based on the algorithm of Alberich-Carrami\~nana, \`Alvarez Montaner and the second author described in \cite{AAD14}.

\subsection{Computation of supercandidates}\label{ssec:Comp_supercandidates}
We will construct a set $S\subset \mathbb Q$ that contains all the jumping numbers, but is in general much smaller than the set of candidate jumping numbers.
Following \cite{AAD14}, we work as follows. Let $S$ be the empty set. One by one, we will add numbers to this set. First, we add the log canonical threshold $\lambda_1=\min\left\{\left.\frac{k_i+1}{e_i}\right| i\in I\right\}$ to $S$. Assume that we added the value $\lambda$, then the next value we add will be
\begin{equation}\label{formula:supercandidate}
\min\left\{\left.\frac{k_i+1+e_i^{\lambda}}{e_i} \right| i\in I \right\},
\end{equation}
where $D_\lambda=\sum_{i\in I} e_i^{\lambda} E_i$ is the $\pi$-antieffective closure of $\lfloor \lambda F\rfloor - K_\pi$.

\begin{defn}
The elements of the set $S$ are called \textit{supercandidates}.
\end{defn}
Following Definition 4.3 in \cite{AAD14}, we define the minimal jumping divisor.
\begin{defn}
If $\lambda$ is a supercandidate, then the \textit{minimal jumping divisor} associated to $\lambda$ is the reduced divisor $G_\lambda$, supported on those components $E_i$ where the minimum in (\ref{formula:supercandidate}) is reached.
\end{defn}
\begin{rmk}
The name \emph{minimal jumping divisor} is in contrast to the \textit{maximal jumping divisor}, which is the reduced divisor supported on all the $E_i$ for which $\lambda$ is a candidate. We will not use this notion in this paper.
\end{rmk}

\begin{thm}
All jumping numbers are supercandidates.
\end{thm}
\begin{proof}

Suppose $\lambda$ is a jumping number, which is not a supercandidate. Let $\lambda'$ be the largest supercandidate smaller than $\lambda$. Note that this $\lambda'$ always exists since $\lct(X,\ak)$ is a supercandidate smaller than $\lambda$, and the supercandidates are a discrete set. Then $\lambda$ is strictly smaller then the supercandidate following $\lambda'$, i.e.,  $\lambda<\min\left\{\left.\frac{k_i+1+e_i^{\lambda'}}{e_i} \right| i\in I \right\}$, where $D_{\lambda'}=\sum_{i\in I} e_i^{\lambda'}$ is the $\pi$-antieffective closure of $\lfloor\lambda'F\rfloor-K_\pi$. But then
\[-D_{\lambda'}\leq K_\pi-\lfloor\lambda F\rfloor \leq K_\pi-\lfloor\lambda' F\rfloor\,,\]
and hence\footnote{Indeed, if $D_1\leq D_2$, then $\Oc_Y(D_1)\subseteq \Oc_Y(D_2)$, pushing forward preserves inclusion, and $D_{\lambda'}$ and $\lfloor\lambda'F\rfloor-K_\pi$ are equivalent.}
\[\mathcal J(X,\ak^\lambda)=\mathcal J(X,\ak^{\lambda'})\,,\]
contradicting the fact that $\lambda$ is a jumping number.
\end{proof}

\begin{rmk}
Note that we cannot conclude as in \cite{AAD14} that every supercandidate is a jumping number, since there is no correspondence between $\pi$-antieffective divisors and integrally closed ideals in higher dimensions. However, no examples are known where not all supercandidates are actual jumping numbers.
\end{rmk}

In the next subsection, we discuss some techniques to check whether supercandidates are jumping numbers.

\subsection{Checking supercandidates}\label{ssec:Check_supercandidates}
Once we have our supercandidates, we have to check whether they are actual jumping numbers. An important tool is the following proposition.

\begin{prop}
 If $\lambda$ is a jumping number, then $G_{\lambda}$ contributes $\lambda$. In particular, there is a divisor $G$ critically contributing $\lambda$, satisfying $G\leqslant G_{\lambda}$. Moreover, if $\lambda'$ is the supercandidate previous to $\lambda$, we have \[\pi_*\Oc_Y(K_\pi-\lfloor \lambda F\rfloor+G_\lambda)=\J(\ak^{\lambda'})\,.\]
\end{prop}

\begin{proof}
 Let $D_{\lambda'}=\sum_{i\in I} e_i^{\lambda'} E_i$ be the $\pi$-antieffective closure of $\lfloor\lambda' F\rfloor-K_\pi$. We claim that $K_\pi-\lfloor \lambda F\rfloor+G_\lambda\geqslant -D_{\lambda'}$. In fact, for any fixed $E_i$, one has by formula (\ref{formula:supercandidate}) that \[\lambda\leqslant\frac{k_i+1+e_i^{\lambda'}}{e_i}\,,\] so \[k_i- \lambda e_i +1\geqslant -e_i^{\lambda'}\,.\]  If $E_i\leq G_\lambda$, this is an equality. Otherwise the inequality is strict, and then, since $-e_i^{\lambda'}$ is an integer, after rounding up we get \[\lceil k_i -\lambda e_i+1\rceil \geq -e_i^{\lambda'}+1\,,\] and hence \[k_i- \lfloor\lambda e_i\rfloor \geq -e_i^{\lambda'}\,.\] We conclude that indeed $K_\pi-\lfloor \lambda F\rfloor+G_\lambda\geqslant -D_{\lambda'}$.
 
Note that also $K_\pi-\lfloor \lambda F\rfloor+G_\lambda \leq K_\pi - \lfloor\lambda'F\rfloor$, because $\lambda$ is a candidate for $G_\lambda$, and hence
\[\pi_*\Oc_Y(-D_{\lambda'})\subseteq\pi_*\Oc_Y(K_\pi-\lfloor \lambda F\rfloor+G_\lambda)\subseteq\pi_*\Oc_Y(K_\pi - \lfloor\lambda'F\rfloor)\,.\]
By Theorem \ref{Thm:unloading}, all these ideals must be the same, and equal to $\mathcal J(\ak^{\lambda'})$. Since $\lambda$ is a jumping number, we can conclude that
\[\mathcal J(\ak^\lambda)\varsubsetneq\pi_*\Oc_Y(K_\pi-\lfloor \lambda F\rfloor+G_\lambda)\,,\]
so $G_\lambda$ contributes $\lambda$.

In particular, there exists a $G\leq G_\lambda$ critically contributing $\lambda$.
\end{proof}

\begin{rmk}\label{rmk:JD_not_connected}
 Here it is important to notice that $G_\lambda$ is not necessarily connected, unlike the critically contributing divisors (see Corollary \ref{cor:connection_critical}).
\end{rmk}

So in order to check whether a supercandidate $\lambda$ is a jumping number, it suffices to check whether some $G\leq G_\lambda$ contributes $\lambda$. By Corollary \ref{cor:connection_critical}, we know that if we can find such a $G$, we can even find a connected one.

The following proposition shows that a supercandidate is a jumping number if its minimal jumping divisor has an irreducible connected component.

\begin{prop}\label{prop:irr_mjd}
If $\lambda$ is a supercandidate such that $G_\lambda$ has a connected component that is irreducible, then $\lambda$ is a jumping number.
\end{prop}
\begin{proof}

From the proof of the previous proposition, using the same notations, we see that \[K_\pi-\lfloor\lambda F\rfloor+G_\lambda = -D_{\lambda'} + \sum_{E_i\not\leq G_{\lambda}} a_i E_i\,,\]
where $a_i\geq 0$. So for $E\leq G_{\lambda}$ we have \[\left.\left(K_\pi-\lfloor\lambda F\rfloor+G_\lambda\right)\right|_E = \left.-D_{\lambda'}\right|_E + \sum_{E_i\not\leq G_\lambda} a_i \left.E_i\right|_E\,,\]
which is effective in $\Pic E$ because $D_{\lambda'}$ is $\pi$-antieffective and $E$ is different from the $E_i$ that appear in the sum.

Now if $E$ is an irreducible connected component of $G_\lambda$, we have \[\left.\left(K_\pi-\lfloor\lambda F\rfloor+G_\lambda\right)\right|_E = \left.\left(K_\pi-\lfloor\lambda F\rfloor+E\right)\right|_E\,,\]
so the fact that this divisor is effective implies by Propositon \ref{prop:contribution_globalsections} that $\lambda$ is a jumping number contributed by $E$.
\end{proof}

So if $G_\lambda$ has an irreducible connected component, there is nothing to check anymore. This is actually a very important case, since in general many supercandidates have an irreducible minimal jumping divisor. In the other case, it would suffice to check whether $G_\lambda$ contributes $\lambda$ as a jumping number. However, in practice, this seems to be hard. Therefore, we suggest the following approach, which is more likely to work. Start by checking contribution by an irreducible $E\leq G_\lambda$. This can be done using Proposition \ref{prop:contribution_globalsections}, if we have enough understanding of $\Pic E$. If this does not give a positive answer, check the connected $G\leq G_\lambda$ consisting of two irreducible components, and proceed in this manner up to the maximal connected divisors $G\leq G_\lambda$. As soon as we find a $G\leq G_\lambda$ contributing $\lambda$, we know that $\lambda$ is a jumping number. If there is no such (connected) $G$, $\lambda$ is not a jumping number.

If $G$ is reducible it is not very clear how to check whether it contributes a supercandidate $\lambda$ as a jumping number. However, we have some tools that can be useful.

\begin{prop}\label{two_divisors_contribute}
Suppose that $\lambda$ is a supercandidate, such that $\lambda$ is a candidate for $G=E_1+E_2$. Suppose that $\Oc_Y(K_\pi-\lfloor\lambda F\rfloor + G)|_{E_i}\cong \Oc_{E_i}$ for $i\in\{1,2\}$. Then $\Oc_Y(K_\pi-\lfloor\lambda F\rfloor + G)|_{G} \cong \Oc_{G}$, and hence $\lambda$ is a jumping number contributed by $E_1+E_2$.
\end{prop}

\begin{proof}

By Corollary \ref{cor:connection_critical}, we can assume that $E_1$ and $E_2$ intersect transversally. Then by the following lemma, we see that \[\Oc_Y(K_\pi-\lfloor\lambda F\rfloor + G)|_{G} \cong \Oc_G,\] and hence $H^0(G,\Oc_Y(K_\pi-\lfloor\lambda F\rfloor + G)|_{G})\neq 0$. By Proposition \ref{prop:contribution_globalsections}, this means that $G$ contributes $\lambda$ as a jumping number.
\end{proof}
\begin{lemma}
Let $G=E_1\cup E_2$ be a closed reducible connected subvariety of a variety $Y$ such that $E_1$ and $E_2$ intersect transversally, and $D:=E_1\cap E_2$ is connected. Let $\mathcal L$ be a sheaf of $\Oc_Y$-modules such that $\mathcal L|_{E_i}\cong \Oc_{E_i}$ for $i\in\{1,2\}$. Then $\mathcal L|_G \cong \Oc_G$.
\end{lemma}

\begin{proof}
Consider the short exact sequence
\[0\to \Oc_G \to \Oc_{E_1}\oplus\Oc_{E_2} \to \Oc_D\to 0\]
on $G$. Here we consider sheaves $\Oc_E$, where $E$ is equal to $E_1$, $E_2$ or $D$, as a sheaf on $G$ by the pushforward through the inclusion morphism $E\subset G$. After tensoring with the restriction of $\mathcal L$ to $G$, we get
\[0\to \mathcal L|_G \to \mathcal L|_{E_1}\oplus\mathcal L|_{E_2} \to \mathcal L|_D \to 0.\]
To compute the second and third term, we used the projection formula: if $\iota: E\to G$ is the inclusion, where $E$ is again equal to $E_1$, $E_2$ or $D$, then $\iota_*\Oc_E\otimes_{\Oc_G} \mathcal L|_G = \iota_*(\Oc_E\otimes_{\Oc_E} \mathcal L|_E) = \iota_*\mathcal L|_E$.

By assumption, there are isomorphisms $\phi_i: \Oc_{E_i}\to\mathcal L|_{E_i}$ for $i\in\{1,2\}$, and restricting to $D$ gives isomorphisms $\phi_i|_D:\Oc_D \to \mathcal L|_D$. Hence $\phi_1|_D^{-1}\circ \phi_2|_D$ is an automorphism of $\Oc_D$, so it corresponds to multiplication with an element $c\in\Gamma(D,\Oc_D^*)=k^*$. Composing $\phi_1$ with multiplication with this constant yields $\phi_1|_D=\phi_2|_D$. 

This means that we have a commutative diagram
\[\begin{CD}
0 @>>> \Oc_G @>>> \Oc_{E_1}\oplus\Oc_{E_2} @>>> \Oc_D @>>> 0\\
@. @. @V{\phi_1\oplus\phi_2}VV @V\phi_D VV\\
0 @>>> \mathcal L|_G @>>> \mathcal L|_{E_1}\oplus\mathcal L|_{E_2} @>>> \mathcal L|_D @>>> 0.
\end{CD}\]
Then we see that $\mathcal L|_G$ must be isomorphic to $\Oc_G$, since they are the kernel of the same morphism.
\end{proof}

\begin{rmk}

The condition in Proposition \ref{two_divisors_contribute} seems quite special, but it is the generalization of the analogous result in the two-dimensional case (\cite[Proposition 4.1]{Tuc10}). Moreover, we did not spot any other behaviour in concrete examples.
\end{rmk}

\begin{prop}\label{Necessary_condition_contribution_reduced_divisor}
If $G=E_1 + \dots + E_m$ is a connected divisor critically contributing $\lambda$ as a jumping number, then \[H^0(E_i,\Oc_Y(K_\pi - \lfloor \lambda F\rfloor + G)|_{E_i})\neq 0\] for all $i\in\{1,\dots,m\}$.
\end{prop}
\begin{proof}
It suffices to prove the statement for $E_1$. Denote $G'=E_2+\dots+E_m$, and $D=E_1|_{G'}$. Then we have a short exact sequence
\[0\to \Oc_{G'}(-D)\to \Oc_G \to \Oc_{E_1} \to 0.\]
As before, we consider all the sheaves as sheaves on $G$ by the pushforward. After tensoring with $\Oc_Y(K_\pi -\lfloor \lambda F\rfloor + G)|_G$, we get
\begin{align*}
0 \to \Oc_Y(K_\pi - \lfloor \lambda F \rfloor + G')|_{G'} \to & \,\Oc_Y(K_\pi -\lfloor \lambda F\rfloor + G)|_G \\ &\to \Oc_Y(K_\pi -\lfloor \lambda F\rfloor + G)|_{E_1}\to 0,
\end{align*}
and taking global sections gives
\begin{align*}
0 \to H^0(G',\Oc_Y(K_\pi -\lfloor \lambda F\rfloor + G')&|_{G'})
\to H^0(G,\Oc_Y(K_\pi -\lfloor \lambda F\rfloor + G)|_G) \\ &\to H^0(E_1,\Oc_Y(K_\pi -\lfloor \lambda F\rfloor + G)|_{E_1}
\end{align*}
Since $G$ contributes critically, $H^0(G',\Oc_Y(K_\pi -\lfloor \lambda F\rfloor + G')|_{G'})= 0$ and $H^0(G,\Oc_Y(K_\pi -\lfloor \lambda F\rfloor + G)|_G) \neq 0$. Therefore \[H^0(E_1,\Oc_Y(K_\pi -\lfloor \lambda F\rfloor + G)|_{E_1})\neq 0,\] which proves the proposition.
\end{proof}

\begin{rmk}
This proposition can be used to decide that a divisor $G\leq G_\lambda$ does not contribute $\lambda$ as a jumping number.
\end{rmk}

%

\begin{rmk}
Our algorithm does not work in general, since we need to have enough understanding of the Picard groups of the exceptional divisors, and it is not always clear how to check the existence of global sections of sheaves on reducible varieties. However, it appears to be not realistic to develop a practical algorithm in full generality, since there is a wide range of possible singularities. 
Nonetheless, our technique seems to suffice in many situations.
\end{rmk}

\begin{rmk}\label{rmk:comparison}


One could also determine the jumping numbers by computing all the candidate jumping numbers, and checking whether they are jumping numbers or not. This can be done by checking whether they are contributed by some divisor for which they are a candidate. However, our algorithm is more efficient in general, since if there are many exceptional divisors in a resolution, the set of candidate jumping numbers is much bigger than the set of supercandidates. Moreover, the minimal jumping divisor can be much smaller than the maximal one, which reduces the amount of possible contributing divisors. Also, Proposition \ref{prop:irr_mjd} implies that several supercandidates do not need to be checked anymore.

On the other hand, algorithms as \cite{BL10} and \cite{Shi11} have to pass by computations of generalized Bernstein-Sato polynomials, which is not easy in general.

\end{rmk}

\section{Examples}\label{sec:examples}

\subsection{Example 1}
Let $D$ be the germ of the surface given by $x(yz-x^4)(x^4+y^2-2yz)+yz^4-y^5=0$ in the local ring at the origin in $\mathbb A^3$. We consider an embedded resolution given by six point blow-ups. We start by blowing up at the origin. Then we blow up at the singular point of the strict transform of $D$. In a third step we blow up at the singular point of the strict transform of $D$ on the intersection of $E_1$ and $E_2$, followed by the intersection point of $E_2$, $E_3$ and the strict transform of $D$. The last step consists of blowing up at the two remaining singular points.

Denoting the resolution by $\pi:Y\to X$, we have \[K_\pi = 2E_1 + 4E_2 + 8E_3 + 14E_4 + 6E_5 + 6E_6\] and
\[F=\pi^*D = D_{aff} + 5E_1 + 9E_2 + 16E_3 + 27E_4 + 11E_5 + 11E_6.\] All the exceptional divisors are projective planes, blown up at at most 4 additional points, so one can completely understand their Picard groups and effective cones. Applying the algorithm of section \ref{sec:algorithm}, we find the following supercandidates and their minimal jumping divisors.

\vspace{2mm}

\begin{center}
{\renewcommand{\arraystretch}{1.3}
\begin{tabular}{|c|c|}
\hline
$\lambda$ & $G_\lambda$\\
\hline
$\frac59$ & $E_2+E_4$\\
$\frac23$ & $E_2+E_4$\\
$\frac{20}{27}$ & $E_4$\\
$\frac{7}{9}$ & $E_2+E_4$\\
$\frac{23}{27}$ & $E_4$\\
$\frac89$ & $E_2+E_4$\\
$\frac{25}{27}$ & $E_4$\\
$\frac{26}{27}$ & $E_4$\\
1 & $E_1+E_2+E_3+E_4+E_5+E_6+D_{aff}$.\\
\hline
\end{tabular}
}
\end{center}

\vspace{2mm}

We verify one supercandidate to illustrate the algorithm. It is easy to see that $\lct(X,D) = \frac59$. We have to compute the $\pi$-antieffective closure of $\left\lfloor\frac59F\right\rfloor-K_\pi = E_2+E_4$.

First, note that $E_1$ is isomorphic to $\mathbb P^2$ blown up at two additional points, the second one infinitely near to the first one. As generators of the Picard group, we denote the pullback of a line by $\ell$, the pullback of the first exceptional curve by $e_1$, and the second exceptional curve by $e_2$. Then $-(E_2+E_4)|_{E_1} = -e_1+e_2$, which is not effective. Hence, in the next step of the unloading, we consider $E_1+E_2+E_4$.

Now we see that $E_3$ is isomorphic to $\mathbb P^2$ blown up at one point. If we denote the generators of the Picard group by $\ell$ (pullback of a line) and $e$ (the exceptional curve), we find $-(E_1+E_2+E_4)|_{E_3} = -2\ell$. Since $E_3|_{E_3} = -\ell-e$, we see that we need to add $E_3$ twice to achieve an effective divisor on $E_3$. So in the next step, we consider $E_1+E_2+2E_3+E_4$.

Continuing in this manner, we add $E_4$ twice, $E_5$ and $E_6$, and we can check that the obtained divisor $E_1+E_2+2E_3+3E_4+E_5+E_6$ is $\pi$-antieffective. This implies that the next supercandidate is \[\min\left\{\frac{0+1+0}{1},\frac{2+1+1}{5},\frac{4+1+1}{9},\frac{8+1+2}{16},\frac{14+1+3}{27},\frac{6+1+1}{11} \right\}=\frac69=\frac{18}{27}=\frac23,\]
and this minimum is achieved for the terms coming from $E_2$ and $E_4$.

\vspace{2mm}

All the supercandidates with irreducible minimal jumping divisor are jumping numbers contributed by $E_4$, by Proposition \ref{prop:irr_mjd}. For the numbers $\frac79$ and $\frac89$, one can see that $(K_\pi -\lfloor\lambda F\rfloor+E_2)|_{E_2}$ is effective, so these numbers are jumping numbers contributed by $E_2$. One can check that the supercandidates $\lambda=\frac59$ and $\lambda=\frac23$ are not contributed by a single exceptional divisor. However, for these numbers, we have $(K_\pi -\lfloor\lambda F\rfloor+E_2+E_4)|_{E_i}=0$ in $\Pic E_i$ for $i\in\{2,4\}$, so, by Proposition \ref{two_divisors_contribute}, they are jumping numbers contributed by $E_2+E_4$. Note that in fact we didn't need to check whether $\frac59$ is a jumping number, since this is the log canonical threshold. Finally, since we are in the case of a divisor, $1$ is always a jumping number.

Existing implemented algorithms (\cite{BL10}) did not give a result after several days of computation.

\subsection{Example 2}\label{example:elliptic_curve}
In this example we show that we don't need to understand the Picard groups and effective cones of all exceptional divisors completely in order to compute the jumping numbers.

Take $X =\Spec \mathbb C[x,y,z]_{(x,y,z)}$ and $D$ the zero locus of $$(x^d+y^d+z^d)^2+g(x,y,z),$$ with $d\geq 3$ and $g(x,y,z)$ a generic homogeneous polynomial of degree $2d+1$.

In the first step of the resolution, we blow up at the origin of $X$. We denote the exceptional divisor by $E_1$. Step 2 consists in blowing up at $k=d(2d+1)$ singular points of the strict transform of $D$; the exceptional divisors are denoted by $E^p_i$. In step 3, we blow up at $C$, the intersection of $E_1$ and $D_{aff}$. This is a curve of genus $g=\frac{1}{2}(d-1)(d-2)$. The exceptional divisor is denoted by $E_2$. The final step of the resolution is blowing up at the intersection of $E_1$, $E_2$ and $D_{aff}$, which is also isomorphic to $C$. We denote the exceptional divisor by $E_3$. We denote the composition of these blow-ups by $\pi:Y\to X$. We have
\begin{align*}
\pi^*D &= D_{aff} + 2dE_1 + (2d+2)\sum_{i=1}^k E_i^p + (2d+1)E_2 + (4d+2)E_3,\\
K_\pi &= 2E_1 + 4\sum_{i=1}^k E_i^p + 3E_2 + 6E_3.
\end{align*}

In order to compute the jumping numbers using the unloading procedure, we need to know the mutual intersections of all components of $\pi^*D$ in the Picard groups of the exceptional divisors, as well as the self-intersections of the exceptional divisors, and we need to determine when a divisor of the form
\[-\left.\left(a_1E_1+a_p\sum_{i=1}^kE^p_i+a_2E_2+a_3E_3\right)\right|_{E}\]
is effective in $\Pic E$, where $E$ varies over all the exceptional divisors. (Note that we take one coefficient for all the $E^p_i$, since everything we will encounter will be symmetrical in the $E^p_i$.) In order to do this, we have to know more about the exceptional divisors. The divisor $E_1\subset Y$ is a projective plane, blown up at $k=d(2d+1)$ additional points. The $E_i^p$ are projective planes, blown up at two additional points, the second center lying on the exceptional curve of the first blow-up. The divisors $E_2$ and $E_3$ are ruled surfaces over $C$.

We show how we can obtain conditions on the $a_j$ from the information on $E_3$. The other exceptional divisors are treated similarly.

Since $E_3$ is a ruled surface over $C$, its Picard group equals $\mathbb Z\oplus p^*(\Pic C)$, where $\mathbb Z$ is generated by a section, say $C'$. Here $p$ denotes the canonical morphism $E_3\to C$. It is not obvious to give a complete description of the effective cone
, but we know that $aC' + p^*\mathfrak d$ is effective if $a\geq0$ and $\mathfrak d$ is effective, or if $a\geq0$ and $\text{\rm deg }\mathfrak d\geq d-2$ (in fact the first one suffices for our purposes), and that it is not effective if $a<0$. It will turn out that this is enough information.
We denote
\begin{align*}
E_1|_{E_3} &= C_1,\\
E_i^p|_{E_3} &= f_i,\\
E_2|_{E_3} &= C_2,\\
D_{aff}|_{E_3} &= C_0,
\end{align*}
where all of the $C_j$'s are sections, and the $f_i$ are fibres. From \cite[Proposition 2.1]{Vey91} we know that
\begin{align*}(4d+2)E_3|_{E_3} &= -D_{aff}|_{E_3} - 2dE_1|_{E_3} - (2d+2)\sum_{i=1}^k E_i^p|_{E_3} - (2d+1)E_2|_{E_3}\\
&= -C_0 - 2dC_1 - (2d+2)\sum_{i=1}^k f_i - (2d+1)C_2.
\end{align*}

In order to have more information about the self-intersection of $E_3$, we compute the pullbacks of some additional divisors. First, if $D_1$ is the divisor given by the zero locus of $x^d+y^d+z^d$, we see that $\pi^*D=D_{1,aff} + dE_1+(d+1)E_2 + (d+1)\sum_{i=1}^k E_i^p + (2d+1)E_3$. Analogously, for a generic plane $H$ through the origin, we have $\pi^*H = H_{aff} + E_1+E_2+\sum_{i=1}^k E_i^p + 2E_3$. Since $D_{1,aff}$ does not intersect $E_3$, and $H_{aff}$ intersects $E_3$ in $d$ fibers, we find
\begin{align*}
(2d+1)E_3|_{E_3} &= -dC_1 - (d+1)C_2 - (d+1)\sum_{i=1}^k f_i,\\
2E_3|_{E_3} &= -C_1-C_2-\sum_{i=1}^k f_i - \sum_{j=1}^d f_j',
\end{align*}
where the $f_j'$ are fibres. One easily verifies that $\sum_{i=1}^k f_i - (2d+1)\sum_{j=1}^d f_j'=0$ in $\Pic E_3$, since it is the pullback of the principal divisor on $C$ defined by the rational function $\frac{g}{\ell^{2d+1}}$, where $\ell$ is the linear polynomial defining $H$. (We considered the curve $C$ embedded in the projective plane $E_1$, before the other blow-ups, with its standard coordinates.)
From this observation, together with the equalities above, one can conclude that $C_0=C_1=C_2$ in $\Pic E_3$ and hence that $E_3|_{E_3} + C_1=-(d+1)\sum_{j=1}^d f_j'$ is the pullback of adivisor of degree $-d(d+1)$.

So the remarks about the effective divisors on $E_3$ above yield that a divisor $$-\left.\left(a_1E_1+a_p\sum_{i=1}^kE^p_i+a_2E_2+a_3E_3\right)\right|_{E_3}$$ is effective if $a_3\geq a_1+a_2$ and $(d+1)a_3> (2d+1)a_p$, and that it is not effective if $a_3<a_1+a_2$.

A similar analysis on the other divisors yields the following conclusions. For $E_1$ we obtain sufficient conditions
\begin{equation*}\left\{\begin{aligned}
(2d+1)a_1 &\geq da_3,\\
(d+1)a_1 &\geq da_p,
\end{aligned}\right.
\end{equation*}
the first of which is a necessary condition.

For the $E_i^p$ we obtain sufficient and necessary conditions
\begin{equation*}\left\{\begin{aligned}
a_p &\geq a_1,\\
a_p &\geq a_2,\\
2a_p &\geq a_3.
\end{aligned}\right.
\end{equation*}

The exceptional divisor $E_2$ is also a ruled surface. Here we have sufficient conditions
\begin{equation*}\left\{\begin{aligned}
&2a_2 \geq a_3,\\
&2(d+1)a_2 \geq (2d+1)a_p,
\end{aligned}\right.
\end{equation*}
the first of which is necessary.

All together, this gives the following set of sufficient conditions for a divisor \[\left(a_1E_1+a_p\sum_{i=1}^kE^p_i+a_2E_2+a_3E_3\right)\] to be $\pi$-antieffective:
\begin{equation*}\left\{\begin{aligned}
a_1 &\leq \frac{1}{2}a_3\leq a_2\leq a_p\\
a_3 &\geq a_1+a_2\\
a_3&\leq 2a_1 + \frac{1}{d}a_1\\
a_p &\leq a_1 + \frac{1}{d} a_1\\
a_p &< a_2 + \frac{1}{2d+1}a_2\\
2a_p &< a_3 + \frac{1}{2d+1}a_3.
\end{aligned}\right.
\end{equation*}
The first three inequalities are necessary.

\begin{prop}
In this example, the set of jumping numbers in $(0,1]$ is
\[A:=A_1\cup A_2\cup A_3,\]
where
\begin{align*}
A_1 &= \left\{\left.\frac{n}{2d}\,\right| 3\leq n< d\right\},\\
A_2 &= \left\{\left.\frac{2n+1}{4d+2}\,\right| d\leq n\leq 2d\right\},\\
A_3 &= \left\{\left.\frac{n}{2d}\,\right| d+3\leq n\leq 2d\right\}.
\end{align*}
\end{prop}

\begin{proof}
Since $d\geq 3$, the log canonical threshold is
\[\lambda_1=\min\left\{\frac{3}{2d},\frac{5}{2d+2},\frac{4}{2d+1},\frac{7}{4d+2},1\right\}=\frac{3}{2d}.\]
This is indeed the minimal value in $A$. Moreover, $G_{\lambda_1}=E_1$ if $d>3$, and $E_1+E_3$ if $d=3$, in which case $\lambda_1=\frac{1}{2}$.

The rest of the proof goes by inductively computing the next supercandidates. Given a supercandidate $\lambda$ in one of the three sets, we compute $\lfloor \lambda \pi^*D\rfloor - K_\pi$, and we unload using the set of necessary and sufficient conditions on the coefficients given above. Then the formula (\ref{formula:supercandidate}) gives the next supercandidate.

We treat the case where $\lambda\in A_1$ to show how the proof works. The other cases are similar. Suppose we know that $\lambda=\frac{n}{2d}$ is a supercandidate, where $3\leq n < d$. We have
\[\lfloor \lambda \pi^*D\rfloor-K_\pi = (n-2)E_1+(n-4)\sum_{i=1}^k E^p_i + (n-3)E_2+(2n-6)E_3.\]
Using the conditions for $\pi$-antieffective divisors described above, we find that the $\pi$-antieffective closure is
\[(n-2)E_1+(n-2)\sum_{i=1}^k E^p_i + (n-2)E_2+(2n-4)E_3.\]
Then the next supercandidate is
\[\lambda'=\min\left\{\frac{n+1}{2d},\frac{n+3}{2d+2},\frac{n+2}{2d+1},\frac{2n+3}{4d+2},1\right\}=\frac{n+1}{2d},\]
which is indeed the next value in $A$.

For the minimal jumping divisor, we have
\[G_\lambda = \begin{cases}
E_1+E_3\text{ if }\lambda=\frac{1}{2}\text{ or }\lambda=1,\\
E_1\text{ if }\lambda\in A_1\text{ or }A_3,\\
E_3\text{ otherwise}.
\end{cases}\]
By Proposition \ref{prop:irr_mjd} and checking that $\left.\left(K_\pi-\left\lfloor\frac12\pi^*D\right\rfloor + E_3\right)\right|_{E_3}$ is effective in $\Pic E_3$, all supercandidates are jumping numbers contributed by $E_1$ or $E_3$.
\end{proof}

\addcontentsline{toc}{chapter}{Bibliography}
\bibliographystyle{amsalpha}

\end{document}